\documentclass[12pt, reqno]{amsart}
\usepackage{amsmath, amsthm, amscd, amsfonts, amssymb, graphicx, color, xypic}
\usepackage[british]{babel}
\newtheorem{theorem}{Theorem}[section]

\newtheorem{proposition}[theorem]{Proposition}
\newtheorem{corollary}[theorem]{Corollary}
\theoremstyle{definition}
\newtheorem{definition}[theorem]{Definition}
\newtheorem{example}[theorem]{Example}

\theoremstyle{remark}
\newtheorem{remark}[theorem]{Remark}
\numberwithin{equation}{section}

\begin{document}
\setcounter{page}{1}

\title[First integral and uniformization over the torus]
{First integral by means of the uniformization over the torus}

\author[O. Galdames-Bravo]{Orlando Galdames-Bravo}

\address{O. Galdames Bravo\\
Departamento de Matem\'aticas\\
Universidad de Valencia\\
Doctor Moliner, 50\\
46100 - Burjassot. Spain.}

\email{\textcolor[rgb]{0.00,0.00,0.84}{orlando.galdames@uv.es}}

\subjclass[2010]{Primary 37B30. Secondary 54H20, 74H20}

\keywords{vector field, index, uniformization, first integral}

\date{\today}

\begin{abstract}
The uniformization of a direction field was defined by Finn (in 1973) for the classification 
of certain differential operators. In the present note we recover the idea of uniformization 
in order to generalize it and apply it to the existence of first integrals. Roughly speaking, 
we say that a plane vector field $X$ on $D\subset\mathbb{R}^2$
is uniformizable over $T^2$ (the torus) if there is a constant vector field $Z$ on $\Delta\subset T^2$ and a 
diffeomorphism $\psi\colon D\to \Delta\subset T^2$ such that $d\psi\circ X = Z\circ\psi$. 
\end{abstract} \maketitle

\section{Introduction and terminology}

The existence of a first integral for a vector field is a classical problem in the theory 
of dynamical systems. One of the main tools to this aim is the study of the rectification 
or uniformization of a vector field. Roughly speaking, we understand as uniformization of a 
vector field with plane domain, as the existence of a diffeomorphic map that brings such a 
vector field to a constant vector field. Local uniformization is a well known property of 
any vector field with mild conditions (see e.g \cite[Theorem 1.2]{Dum-Lli-Art}). 
Regarding global plane domains, there are several results over time that deal with several 
techniques in order to find a first integral or solve a related problem, see for instance 
\cite{Bou, Kamke, Kaplan, Mazzi1, Muller, Weiner}. \\

The technique we introduce in the present note is based in the ideas presented by Finn in 
\cite{Finn} (who studied the behavior of the index of a plane direction field along a Jordan 
curve that represent the boundary of the domain) and the suggestion of Mart\'inez-Alfaro 
for the generalization of this concept to other not plane manifolds. In \cite{Finn}, the 
author studied direction fields which have index $0$ or $2$ along suitable curves. Thanks 
to the generalization we give we can extend this study to the index $1$. In brief, a vector 
field is uniformizable the sense of Finn, if it is conjugate with a constant one, both 
plane vector fields. In our case we assume that the constant vector field can be defined 
over the torus. We will see that this point of view will allows us to obtain an application 
to the existence of a first integral for certain vector fields. \\

We have organized the paper in a unique section where we develop the ideas described above.
Let us introduce some terminology for the paper. 
Let $M$ be a manifold, we denote by $T(M)$ he bundle of tangent vector spaces of $M$.
Let $D\subseteq\mathbb{R}^2$ be an closed plane domain, recall that a vector field $X\colon D\to T(D)$ is 
said to be uniformizable if it is conjugate with a constant vector field $U\colon E\to T(E)$ in the plane, 
for a certain plane domain $E$, i.e. if there exists a diffeomorphic map $\varphi\colon D\to E$ so that 
$d\varphi \circ X = U \circ \varphi$ (see \cite[Definition 3]{Finn} and \cite[Section 1.3]{Dum-Lli-Art}). 
We denote the torus by $T^2$. The definition of the index of a direction field 
along a curve can be found in \cite[Section 1]{Finn}. Let us define it for a vector field.
We assume that a Jordan curve is an smooth injective closed curve. As usual we sometimes 
identify the parametrization of a curve with its image. Let $D$ be a plane domain and let 
$\Gamma\colon[0,2\pi]\to D$ be a Jordan curve on $D$. 
Let $X_\Gamma$ be the vector field defined on $\Gamma$, and $\delta(X_\Gamma)$ denotes the 
change of angle from the vector $X_\Gamma(0)$ to the vector $X_\Gamma(2\pi)$. 
Let $X$ be a vector field defined on $D$, we define the \emph{index of the 
vector field $X$ along the curve $\Gamma$} as 
$$
\operatorname{ind}_{\Gamma}(X) := \cfrac1{2\pi} \,\delta(X_\Gamma)\,.
$$
A well known formula to compute this index is: 
\begin{equation}\label{eqindex}
\operatorname{ind}_{\Gamma}(X) = \cfrac1{2\pi} \int_0^{2\pi} 
\cfrac{\det({X|}_{\Gamma},({X|}_{\Gamma})')}{\|{X|}_{\Gamma}\|^2} \,dt ,
\end{equation}
see for instance \cite{Dru-Fed}. A parametrization of the vector field can be 
written as $X = P \frac{\partial}{\partial x} + Q \frac{\partial}{\partial y}$, where 
$P,Q\colon D\to \mathbb{R}$ and $x,y$ are suitable coordinates of $D$. 
Let $h\colon D\to \mathbb{R}$ be a differentiable function.
We say that $h$ is a \emph{first integral of $X$} if $X(h) = P h_x + Q h_y = 0$, 
where $h_x$ and $h_y$ denotes the respective derivatives.
We will mention a suitable domain called \emph{corona}, which is 
a set of the form $\{(x,y)\in\mathbb{R}^2 : 0 < r < \|(x,y)\|_2 < R <\infty\}$.
We denote by $S^n$ the sphere of dimension $n\in\mathbb{N}$. We will identify 
$T^2$ with $S^1\times S^1$ and with $\mathbb{R}^2/[0,2\pi]^2$. The \emph{type} $(p,q)$ of a Jordan curve on the torus 
($S^1\times S^1$) is the number of laps around each circumference $S^1$. 
We can see the type as an element of the homotopy group $\pi_1(T^2)$.
We refer to \cite{Arn,Dum-Lli-Art} for notation, definitions and basic 
results that not appear here. 

\section{First integrals and uniformization over the torus}

In this unique section we introduce the concept of uniformization 
over the torus and illustrates by means of an example that such a concept
is useful for the existence of a first integral of a given vector field.
Then, we give a condition for the uniformization over the torus of certain 
vector fields that involves the index along the boundary of the domain 
of these vector fields. We will assume that a domain is a bounded and connected 
region delimited by Jordan curves. We also assume that every function and 
vector field will be sufficiently differentiable.

\begin{definition}
Let $r\ge 1$ and $D\subseteq\mathbb{R}^2$ be a bounded and
connected domain. Let $X\colon D\to T(D)$ be a  vector field. 
Let $\Delta\subseteq T^2$ be a domain. We say that $X$ is \emph{uniformizable 
over the torus} if there exists $\Delta\subseteq T^2$ such that $X$ is conjugate 
to a constant vector field defined on $\Delta$, i.e. there exists a constant 
vector field $Z\colon\Delta\to T(\Delta)$ and a diffeomorphism $\psi\colon D\to\Delta$ 
such that $d\psi\circ X = Z\circ\psi$.
\end{definition}
In the present paper we just deal with the definition above, over the torus, 
but one can easily extend it to any other manifold.
 
\begin{example}
The constant vector field $X = a \;\partial/\partial x + b \;\partial/\partial y$, 
where $a,b\in\mathbb{R}$ so that $a/b\in\mathbb{Q}\setminus\{0\}$, 
defined in a domain of $\mathbb{R}^2$, is a clear example of uniformizable 
vector field over the torus and over the cylinder.
\end{example}

Let us give an application to the existence of a first integral.
Observe that we also can apply the following to the Finn uniformization.  
\begin{proposition}
Let $D\subseteq\mathbb{R}^2$ be a compact connected domain and $X$ be a vector field on $D$
with Hausdorff flow. Assume that $X$ is uniformizable over the torus, then there is a 
first integral for $X$.
\end{proposition}
\begin{proof}
By hypothesis, there is a diffeomorphism $\psi=(f,g)$ through which 
$X$ is conjugate to a constant vector field 
$Z := a \frac{\partial}{\partial u} + b\frac{\partial}{\partial v}$
on the torus, where $a,b\in\mathbb{R}^r\setminus\{0\}$ satisfy that 
$a/b$ (or $b/a$) is a rational number. Moreover, the partial differential equation  
\begin{equation}\label{eq1}
\left\{
\begin{array}{l}
h_x = b \,f_x - a \,f_y \\
h_y = b \,g_x - a \,g_y
\end{array}
\right.
\end{equation}
has solution (see for instance \cite{Mall}). Let $h$ be such a 
solution. We claim that $h$ is a first integral of $X$.

Let us denote by $(u,v)$ the coordinates in $T^2$ and by $(x,y)$ the coordinates 
in $D$. By hypothesis we know that $d\psi\circ X = Z \circ\psi$. Let us denote 
$X = P(x,y) \frac{\partial}{\partial x} + Q(x,y) \frac{\partial}{\partial y}$, 
where $\dot{x} = P(x,y)$ and $\dot{y} = Q(x,y)$.
On the one hand, since $Z$ is constant, there exists $H$ a first integral of $Z$. 
We can choose, for instance, $H(u,v) = b u - a v$, since $H_u = b$ and $H_v = -a$ and 
$Z(H) = a H_u + b H_v = 0$. On the other hand 
we can choose the variables $u$ and $v$ so that $\psi^{-1}(u,v) = (x,y)$, hence 
\[
d\psi\circ X \circ \psi^{-1}= (f_x \dot{x} + g_x \dot{y}) \frac{\partial}{\partial u} + (f_y \dot{x} + g_y \dot{y}) \frac{\partial}{\partial v} .
\]
Therefore
\begin{equation}\label{eq2}
\begin{split}
d\psi\circ X \circ \psi^{-1} (H) & = (f_x \dot{x} + g_x \dot{y}) H_u + (f_y \dot{x} + g_y \dot{y}) H_v \\
& = \dot{x}(H_u f_x + H_v f_y) + \dot{y}(H_u g_x + H_v g_y) \\
& = \dot{x}(b f_x - a f_y) + \dot{y}(b g_x - a g_y) = 0 .
\end{split}
\end{equation}
We know that $h$ is solution of \eqref{eq1}, hence by \eqref{eq2} 
\[
X(h) = \dot{x} h_x + \dot{y} h_y = \dot{x}(b \,f_x - a \,f_y) + \dot{y}(b \,g_x - a \,g_y) = 0.
\]
That means that $h$ is a first integral for $X$.
\end{proof}

\begin{example}
Let $D\subseteq\mathbb{R}^2$ be a plane domain and $X\colon D\to T(D)$ be vector field.
Let $(a,b)\in\mathbb{R}^2\setminus 0$ such that $a/b$ (or $b/a$) is rational.
Assume that there is a diffeomorphism $\psi \colon D \to \Delta\subseteq T^2$ as the one defined 
above $\psi := (f,g)$ and assume that $X:=P\frac{\partial}{\partial x} + Q\frac{\partial}{\partial y}$ 
is such that
$$
P := \dfrac{a g_y - b f_y}{f_x g_y - f_y g_x}
$$
and
$$
Q := \dfrac{b f_x - a g_x}{f_x g_y - f_y g_x},
$$
where the numerators does not vanishes at the same time in $D$, so the 
vector field will be regular. Observe that the denominator 
is the Jacobian of $\psi$, hence it not vanishes in $D$. This vector field is uniformizable 
by means of the constant vector field $a \frac{\partial}{\partial u} + b \frac{\partial}{\partial v}$, 
where $u,v$ denotes the coordinates in the torus.
\end{example}

Proposition and example above illustrates that uniformization over the torus 
could be a tool in order to find out the existence of a first integral, in 
consequence it is interesting to find conditions related to such a property. 
In the following theorem and corollary we show that there exist certain uniformizables 
vector fields such that the index along suitable Jordan curves, that we can take closer 
to the the boundary of its domain, is equal to $1$. Later, we easily deduce a corollary 
for a more general class of vector fields.\\

Let $Z\colon T^2\to T(T^2)$ be a constant vector field and $\Gamma$ be 
a Jordan curve of type $(p,q) \in\mathbb{Z}^2\setminus\{0\}$ on the torus. 
Clearly $\operatorname{ind}_\Gamma Z = 0$. Then we consider the set 
$T^2\setminus \Gamma_0$ for a suitable Jordan curve with the same type as 
$\Gamma$ and transform it diffeomorphically in a corona on $\mathbb{R}^2$. 
Next result states that the index of the image of such a constant vector 
field along the image of $\Gamma$ is $1$.
\begin{theorem}\label{thm1}
Let $\Gamma$ be a  Jordan curve on the torus $T^2$ with type 
$(p,q)\in\mathbb{Z}^2\setminus\{0\}$. Let $Z\colon T^2\to T(T^2)$ 
be a constant vector field. Then, there exist a Jordan curve $\Gamma_0$ 
such that $\Gamma\cap\Gamma_0 = \emptyset$, a domain $D\in \mathbb{R}^2$ 
and a diffeomorphism $\varphi\colon  T^2\setminus\Gamma_0\to D$ such that 
$$
\operatorname{ind}_{\varphi\circ\Gamma}(d\varphi\circ Z\circ \varphi^{-1}) = 1\,.
$$
\end{theorem}
\begin{proof}
We will use the representation of the torus by means of the usual equivalence 
relation in $[0,2\pi]^2$. Since the type of a curve is invariant by diffeomorphism, 
without loss of generality, we can assume that $\Gamma(t) = (pt,qt)$ for $t\in[0,2\pi]$. 
Again, without loss of generality we take the constant 
vector field $Z:=\partial/\partial x$. Thus, $\Gamma$ has type $(p,q)$ 
and the representation of $\Gamma$ in $\mathbb{R}^2/[0,2\pi]^2$ is the segment between 
$(0,0)$ and $(2p\pi,2q\pi)$ in $\mathbb{R}^2$. The angle of this segment with respect to the 
positive abscissa is $\arctan(\frac{q}{p})$. \\

Let the Jordan curve $\Gamma_0(t) := (pt+1/2,qt+1/2)$, clearly $\Gamma\cap\Gamma_0 = \emptyset$. 
We define the diffeomorphism $\varphi\colon T^2\setminus\Gamma_0 \to \mathbb{R}^2$ given by 
\begin{equation}\label{eq3}
\varphi := \sigma\circ \tau\circ \rho_{p,q}\,,
\end{equation}
where 
$$
\rho_{p,q}:=\dfrac1{p^2+q^2}
\Big(\begin{matrix}
p & q\\
-q & p
\end{matrix}\Big)
$$ 
is the clockwise rotation centered at the origin and angle 
$\arctan(\frac{q}{p})$, which move $\Gamma$ to the segment $[0,2\pi \sqrt{p^2+q^2}]\times\{0\}$; 
$\tau$ is the translation 
$$
\tau(x,y)=(x,y+2\pi) ,
$$ 
which translate the segment in the plane 
$[0,2\pi \sqrt{p^2+q^2}]\times\{0\}$ to the segment 
$[0,2\pi \sqrt{p^2+q^2}]\times\{2\pi\}$; 
and $\sigma$ is the change to polar variables, that is 
$$
\sigma(x,y)=(y\cos(x),y\sin(x)).
$$ 
Therefore 
$$
\varphi(x,y)=\dfrac{1}{p^2+q^2}\big((py-qx+2\pi)\cos(px+qy)\,,\,(py-qx+2\pi)\sin(px+qy)\big)\,.
$$
We now use the formula \eqref{eqindex} in order to obtain the index 
of the vector field $d\varphi \circ Z \circ \varphi^{-1}$ along 
$\varphi \circ \Gamma$ by a direct computation. 
On the one hand we have that 
\begin{equation}\label{eqcosi}
(py-qx)|_{\Gamma(t)}=p(qt)-q(pt)=0 \quad\mbox{ and }\quad (px+qy)|_{\Gamma(t)}=(p^2+q^2)t\,.
\end{equation}
For the aim of brevity we define $H:=p^2+q^2$, $c := \cos(Ht)$ and $s := \sin(Ht)$. Hence 
$$
\varphi\circ\Gamma = \cfrac{2\pi}{H}(c,s)\,,
$$
defines a curve on a certain domain $D\subset \mathbb{R}^2$. 
On the other hand we get the following commutative diagram
$$
\xymatrix{
[0,2\pi]\ar[r]^\Gamma\ar[rd]_{\varphi \circ \Gamma}& T^2\ar[r]^Z\ar[d]_{\varphi}&T( T^2)\ar[d]^{d\varphi}\\
&\quad D\subset\mathbb{R}^2\ar[r] & T(D)
}
$$
Therefore,  
\begin{equation}\label{eqvect}
{(d\varphi \circ Z\circ \varphi^{-1})|}_{\varphi\circ\Gamma} = {d\varphi \circ Z|}_{\Gamma} 
= {d\varphi(\partial/\partial x)|}_{\Gamma} = {\dfrac{\partial \varphi}{\partial x}|}_{\Gamma}\,.
\end{equation}
The partial derivative of $\varphi$ with respect to $x$ is 
\begin{equation*}
\begin{split}
\dfrac{\partial \varphi}{\partial x}(x,y) = \cfrac{1}{H} & \big(-q\cos(px+qy)-p(py-qx+2\pi)\sin(px+qy),\\
&-q\sin(px+qy)-p(py-qx+2\pi)\cos(px+qy)\big)\,.
\end{split}
\end{equation*}
Thus, taking into account \eqref{eqcosi} and the given notation, we obtain that 
$$
{\tfrac{\partial \varphi}{\partial x}|}_{\Gamma} = \dfrac{1}{H}(-q c - 2 \pi p s, -q s - 2 \pi pc)\,.
$$
Now, having in mind that $c' = -H s$ and $s' = H c$, we can compute the derivative 
$$
({\tfrac{\partial \varphi}{\partial x}|}_{\Gamma})' = (q s + 2 \pi pc, -q c - 2\pi ps)\,.
$$
Finally, we have that the square norm is 
$$
{\big\| {\tfrac{\partial \varphi}{\partial x}|}_{\Gamma} \big\|}_2^2 = \frac{1}{H} \big((q c+2 \pi ps)^2 + (q s+2 \pi pc)^2 \big)\,, 
$$
in consequence 
$$
\det({\tfrac{\partial \varphi}{\partial x}|}_{\Gamma},({\tfrac{\partial \varphi}{\partial x}|}_{\Gamma})') = 
\cfrac1{H} \big((q c+2p \pi s)^2 + (q s+2p \pi c)^2 \big) = {\big\| {\tfrac{\partial \varphi}{\partial x}|}_{\Gamma} \big\|}_2^2\,. 
$$
This equality together with \eqref{eqvect}, implies that
$$
\operatorname{ind}_{\varphi\circ\Gamma}(d\varphi \circ Z\circ \varphi^{-1}) 
= \cfrac{1}{2\pi} \int_0^{2\pi} \frac{\det({\tfrac{\partial \varphi}{\partial x}|}_{\Gamma},
({\tfrac{\partial \varphi}{\partial x}|}_{\Gamma})')} {{\big\| {\tfrac{\partial \varphi}{\partial x}|}_{\Gamma} \big\|}_2^2} \, dt = 1\,, 
$$
which is the claim of the theorem.
\end{proof}

\begin{remark}
Observe that the domain $D$ given in the proof of theorem above can be taken  
diffeomorphic to the interior of the corona $\{(x,y)\in\mathbb{R}^2 : 1\le\sqrt{x^2+y^2} \le2\}$.
Thus any vector field defined on a compact domain dipheomorphic to $C$, defined 
by two contours $\Gamma_1$ and $\Gamma_2$, so that it is uniformizable over the 
torus by means of the inverse of the map $\varphi$ defined in \eqref{eq3} is a 
good candidate that could satisfies that $\operatorname{ind}_{\Gamma_i} X = 1$, 
for $i=1,2$.\\
\end{remark}
\begin{remark}
One can see that our approach to the study of the index is quite different to the 
one given by Finn in \cite{Finn}. There are two main reasons that bring us to 
change the original point of view: First, the notion of index along a curve has 
evolved, in fact, sometimes it is called winding number of a vector field (see e.g. 
\cite{Dru-Fed}). Second, some results given in the original paper, although they were 
correct, they were based (a little) on the intuition.
\end{remark}

In our last result we provide a necessary condition for the uniformization over 
the torus for suitable vector fields. 
\begin{corollary}
Let $X$ be a  vector field on a bounded and connected domain. Assume that 
$X$ is uniformizable over the torus through a domain in $T^2$ 
which is diffeomorphic to $T^2\setminus \Gamma_0$ for a 
Jordan curve $\Gamma_0$. Then, for each Jordan curve $\Gamma$ 
with the same type as $\Gamma_0$ satisfying that $\Gamma\cap\Gamma_0=\emptyset$, 
there exists a diffeomorphism $\psi$ such that
$$
\operatorname{ind}_{\varphi\circ\Gamma}(d\psi\circ X\circ \psi^{-1}) = 1\,.
$$
\end{corollary}
\begin{proof}
The proof is easy to deduce taking into account the following diagram
$$
\xymatrix{
D\ar[r]^\varphi \ar[d]_X & \Delta \ar[r]^{\!\!\!\!\!\alpha}\ar[d]_Z & T^2\setminus\Gamma_0\ar[r]^{\;\;\beta}\ar[d]_{\widehat{Z}} & C \ar[d]_U\\
T(D) \ar[r]^{d\varphi} & T(\Delta) \ar[r]^{\!\!\!\!\!d\alpha} & T(T^2\setminus\Gamma_0) \ar[r]^{\;\;\;d\beta} & T(C) \,,\\
}
$$
where $Z$, $\widehat{Z}$, $\varphi$ and $\alpha$ are the maps defined 
in the hypothesis and $U$ and $\beta$ are defined in Theorem \ref{thm1}.
Observe that we can choose $\widehat{Z}$ as a constant vector field, since 
$Z$ is constant. Now, having in mind that the vector field 
$\widehat{Z} = d\alpha\circ d\varphi\circ X \circ \varphi^{-1}\circ\alpha^{-1}$  
is constant and applying theorem above for each $\Gamma$ as in the hypothesis, 
we obtain the result.
\end{proof}

Despite we cannot give a sufficient condition, the condition  
for the index is meaningful and provides a good clue for the wanted uniformization.\\

{\bf Acknowledgment.} The original idea of a generalization of the uniformization in the sense 
of Finn was suggested by Professor Jos\'{e} A. Mart\'{\i}nez Alfaro around the year 2010.
I am very grateful for his support and encouragement along those years.

\bibliographystyle{amsplain}

\end{document}